\documentclass[11pt]{amsart}

\usepackage{subfiles}  

\usepackage{geometry}
\usepackage{fullpage}
\usepackage{amstext}
\usepackage{amsthm}
\usepackage{amssymb}
\usepackage[colorlinks=true, linkcolor=blue, citecolor=green, urlcolor=cyan]{hyperref}

\usepackage{tikz}
\usetikzlibrary{mindmap}
\usepackage{subfigure}

\usepackage{ulem}

\usepackage[ruled]{algorithm2e}

\newtheorem{theor}{Theorem}[section]
\newtheorem*{theorem*}{Theorem}

\newtheorem{lemma}[theor]{Lemma}
\newtheorem*{lemma*}{Lemma}

\newtheorem{defi}[theor]{Definition}

\theoremstyle{definition}
\newtheorem{rem}[theor]{Remark}

\theoremstyle{plain}

\newcommand{\R}{\mathbb{R}}

\def\Prob{{\mathbb P}}

\usepackage{xcolor}
\definecolor{b}{HTML}{4472c4}
\definecolor{o}{HTML}{ED7D31}
\definecolor{g}{HTML}{70ad47}
\definecolor{t}{RGB}{40,154,150}

\def\R{{\mathbb R}}

\def\Prob{{\mathbb P}}
\def\Exp{{\mathbb E}}

\def\Vol{{\rm Vol}}

\newcommand{\norm}[1]{\left\| #1 \right\|}

\newcommand{\step}[1]{\medskip  \noindent \underline{#1.\,}}

\newcommand{\Qs}{  Q_1^\circ}
\newcommand{\Cu}{\mathcal{C}_+}
\newcommand{\Cm}{\mathcal{C}_-}

\date{\today}
\title{Hardness of approximation of centered convex bodies by polytopes}

\author{Han Huang and Mark Rudelson}

\address{Department of Mathematics, University of Missouri}
\email[H.~Huang]{hhuang@missouri.edu}

\address{Department of Mathematics, University of Michigan}
\email[M.~Rudelson]{rudelson@umich.edu}

\begin{document}
\begin{abstract}
    The distance between convex bodies $K, L \subseteq \R^n$ is defined as 
\[
  d(K,L)= \inf \left\{ \lambda \ge 1: \  L-x \subseteq T (K-y) \subseteq \lambda (L-x) \right\},
\]
where the infimum is taken over all $x,y \in \R^n$ and all invertible linear operators $T: \R^n \to \R^n$.
If both bodies are centrally symmetric, then the shifts $x$ and $y$ can be chosen to be $0$. In this case, any convex symmetric body $K$ can be approximated by a polytope $P$ with at most $N \in (n, e^{cn})$ vertices so that 
\[
 P \subseteq K \subseteq \lambda P
\]
where $\lambda= O \left(\sqrt{\frac{n}{\log N}} \right)$ up to logarithmic factors.

We prove that approximating a general centered convex body by a polytope requires a significantly larger number of vertices compared to the symmetric case.
More precisely, there exists a convex body $K \subseteq \R^n$ whose barycenter coincides with the origin, such that any polytope $P$ satisfying
\[
 P \subseteq K \subseteq c \, \frac{n}{\log N} P
\]
must have at least $N$ vertices, provided that $N \in (Cn^2, e^{cn})$.

Moreover, we prove that the same bound holds for approximating a centered convex body with a polytope having $N$ facets instead of $N$ vertices.
\end{abstract}

\maketitle

\section{Introduction}

The study of the Banach-Mazur distance between convex symmetric bodies in $\R^n$ is among the central topics in geometric functional analysis due to  such bodies being the  unit balls of normed spaces; see, e.g., \cite{Pisier1989, T-J1989}. 
The distance between convex symmetric bodies $K$ and $L$ is defined as
\[
  d(K,L)= \inf_{T \in Gl(n)} \{ \lambda \ge 1: \ L \subseteq TK \subseteq \lambda L \},
\]
where the infimum is taken over all invertible linear mappings $T$. 
The breadth of the topic does not allow us to give a concise review, so we will mention only two directions: the maximal distance and the approximation of such bodies by polytopes with  few vertices or facets.
A classical theorem of John \cite{Pisier1989} asserts that for any convex symmetric bodies $K,L \subseteq \R^n$, $d(K,B_2^n), d(L,B_2^n) \le \sqrt{n}$ which implies that 
\[
  d(K,L) \le n.
\]
In 1988, Gluskin proved that this estimate is essentially optimal by constructing two random convex symmetric polytopes with distance at least $cn$ \cite{Gluskin1988}. 
Here and below, $C,c$, etc. stand for absolute constants. 

The problem of approximating convex symmetric bodies by polytopes with few vertices or facets was resolved in the last 15 years. 
In addition to its theoretical value, this problem is important because of its applications in computer science \cite{B-V2008}. 
Barvinok \cite{Barvinok2014} proved that for any $N >Cn$ and for a given convex symmetric body $K \subseteq \R^n$, there exists a symmetric polytope $P_N \subseteq \R^n$ with $N$ vertices such that 
\begin{equation}  \label{eq: upper symm}
  d(K,P_N) \le C \max \left( 1, \sqrt{\frac{n}{\log N} \log \frac{n}{\log N}} \right).  
\end{equation}

This bound is almost optimal even for the Euclidean ball $K=B_2^n$. Namely, if $N \ge cn$, then \cite{BLM89, C-P1988, Gluskin-par1988}
\[
  d(B_2^n, P_N) \ge c \sqrt{\frac{n}{\log (N/n)}}. 
\]

Similar questions for general convex bodies turned out to be significantly more difficult.
As the origin plays no special role in the non-symmetric setting, we  allow shifts in the definition of the distance between general convex bodies.
More precisely, for  $K,L \subseteq \R^n$ we set
\[
  d(K,L)= \inf_{x,y \in \R^n, \ T \in Gl(n)} \left\{ \lambda \ge 1: \  L-x \subseteq T (K-y) \subseteq \lambda (L-x) \right\}.
\]
John's theorem holds for general convex bodies as well, albeit with a different bound. That is, if $K,L \subseteq \R^n$ are convex bodies, then $d(K,B_2^n), d(L,B_2^n) \le n$, which yields $d(K,L) \le n^2$. 
However, unlike  in the symmetric setting, this bound is suboptimal.
A breakthrough paper \cite{Bizeul-Klartag2025} by Bizeul and Klartag proved that 
\[
  d(K,L) \le C n \log^a n
\]
for some absolute constant $a>0$, which greatly improved the previously known bound $O(n^{4/3} \log^a n)$ \cite{R2000}.
This result shows that the maximal distance between general convex bodies is roughly the same as between symmetric ones. 
Moreover, to evaluate the distance between $K$ and $L$, it is enough to put both bodies in the isotropic position. This, in particular, means that the shifts $x$ and $y$ can be chosen as the barycenters of  $K$ and $L$ respectively.

The problem of approximating a convex body by a polytope with few vertices remains open.
In what follows, we concentrate on the coarse approximation when the number of vertices of the approximating polytope is at most exponential in dimension. For the fine regime $P_N \subseteq K \subseteq (1+\varepsilon)P_N$ that requires at least an exponential number of vertices, see \cite{NNB2020}.
If the body $K$ is shifted so that its barycenter coincides with the origin, then for any $N$ such that $Cn \le N \le e^{cn}$, there exists a polytope $P_N$ with $N$ vertices satisfying
\begin{equation}  \label{eq: BCH inclusion}
    P_N \subseteq K \subseteq C \frac{n}{\log N} P_N,
\end{equation}
and so 
\begin{equation}  \label{eq: upper non-sym}
     d(K,P_N) \le C \frac{n}{\log N}.
\end{equation}
This bound was obtained by Brazitikos et. al. \cite{BCH17} by constructing $P_N$ as the convex hull of $N$ independent random vectors uniformly distributed in $K$.
The probability estimate was later improved in \cite{Nas19}.
The same bound on the distance was also obtained by Szarek \cite{Szarek2014} who used the greedy algorithm to construct $P_N$ and a different shift.

Comparison between \eqref{eq: upper symm} and \eqref{eq: upper non-sym} reveals a large gap between the known bounds in the symmetric and general case. Moreover, it is unknown whether one can construct a polytope $P$ with a polynomial in $n$ number of vertices such that $d(K,P) =O(n^{1-a})$ for some absolute constant $a>0$.
Obtaining a tight estimate for the distance to a polytope requires choosing appropriate shifts, which is a major factor contributing to the difficulty of the problem compared to the symmetric setting.

The only known result in the negative direction belongs to the first author, who showed that if $c \sqrt{n} \le R \le c'n$ then  there exists a convex body $K \subseteq \R^n$ whose John's ellipsoid is centered at the origin such that any polytope satisfying
\[
 K \subseteq P \subseteq RK
\]
must have at least $\exp \left(C \log \big( \frac{R^2}{n}  \big) \frac{n^2}{R^2} \right)$ vertices \cite{Huang2019}. 
On the other hand, as \cite{Huang2018} shows, the center of the John ellipsoid and the barycenter can be extremely far apart. 
Namely, there exists a convex body $K \subseteq \R^n$ whose John's ellipsoid $\mathcal{E} \subseteq K$ is centered at the origin and whose barycenter is not contained in $(1-o_n(1))n \mathcal{E}$, while $K \subseteq n \mathcal{E}$ from John's theorem.

Since by a theorem of Bizeul and Klartag, centering bodies at their barycenters allows obtaining an approximately optimal bound for the maximal distance, it is of interest to check whether the barycenter would provide a good center for polytope approximation as well.
This question was the main motivation for this paper. 
Theorem \ref{thm: main} gives a strong negative answer to this question. 
Moreover, it shows that one cannot obtain a better approximation with another classical center, namely the Santal\'{o} point.
For a convex body $K \subseteq \R^n$ containing the origin in its interior, define the polar body  as
\[
  K^\circ= \{y \in \R^n: \ \langle y,z \rangle \le 1 \text{ for all } z \in K \}.
\]
The Santal\'{o} point of $K$ is defined as the unique point $x \in K$ at which the function
\[ 
  \Vol(K-x) \cdot \Vol \left( (K-x)^\circ \right)
\]
attains its minimal value; see \cite{Schneider1993}. 
If the body $K$ is shifted to put its Santal\'{o} point at the origin, the body $K^\circ$ is centered, i.e., its barycenter is also at the origin.
This makes the Santal\'{o} point another natural candidate for the center of approximation.

As our main result shows, centering the body neither at the barycenter nor at the Santal\'{o} point can yield a better approximation compared to \eqref{eq: upper non-sym}.
More precisely, we have the following theorem.

\begin{theor} \label{thm: main} 
    Let $n$ be sufficiently large, and assume that $N$ satisfies
    \begin{equation} \label{eq: l_range_main}
        C_{\ref{thm: main}} n^2 \le N \le \exp\left(\frac{n}{C_{\ref{thm: main}}}\right).  
    \end{equation}
    Then there exists a convex body $K \subseteq \R^n$ with the following property: \\   
    for any vector $\xi \in \R^n$ lying on the line segment connecting the barycenter of $K$ and its Santaló point, any polytope $P$ satisfying
    \begin{equation*}
        P \subseteq K -\xi \subseteq \frac{n}{C_{\ref{thm: main}}\log N}P
    \end{equation*}
    must have at least $N$ vertices.
\end{theor}

In the case where the body is centered, i.e., its barycenter is the origin, the lower bound for the number of vertices matches \eqref{eq: BCH inclusion} for any $N$ satisfying \eqref{eq: l_range_main}.
This leads to two possibilities, each of which would have interesting implications. 
The first is that the estimate \eqref{eq: upper non-sym} is the best possible with any choice of center, meaning that the approximation of general convex bodies by polytopes requires many more vertices compared to symmetric ones.
Alternatively, it is possible that unlike in \cite{Bizeul-Klartag2025}, the choice of the barycenter or the Santal\'{o} point as the origin is not optimal, and there exists a new, currently unknown center that yields a better approximation.

  A well-known problem in computer science asks about approximation of a convex body by a polytope with a few facets, see, e.g., \cite{B-V2008}.
  This setup arises naturally since convex bodies are feasible sets of convex programming problems, and a polytope in $R^n$ is canonically represented as a set of solutions of a system of linear inequalities
  \begin{align*}
      A x &= b, \\
      x &\ge 0,
  \end{align*}
  where $x \in \R^n, \ b \in \R^N$, and $A$ is an $n \times N$ matrix.
  In this case, $N$ defines the complexity of the problem, and it is desirable to minimize it while maintaining the required precision of the approximation. 

  Our second result, Theorem \ref{thm: main-dual}, provides the same lower estimate for the distance of a centered convex body to a polytope with few facets as Theorem \ref{thm: main} yielded for approximation by a polytope with few vertices.
  It cannot be derived from Theorem \ref{thm: main} by taking polars. 
  Indeed, if we choose a body $K$ constructed in Theorem \ref{thm: main} and select $\xi_1$ as its Santal\'{o} point, then $(K-\xi_1)^{\circ}$ would be an example for Theorem \ref{thm: main-dual} and would have the barycenter at the origin. At the same time, if we use the same body $K$ and choose $\xi_2$ as its barycenter, then $(K-\xi_2)^\circ$ will have the Santal\'{o} point at the origin. 
  However, $(K-\xi_1)^{\circ}$ and $(K-\xi_2)^{\circ}$ are not shifts of each other, since taking the polar does not commute with the shift.

  Nevertheless, we can prove the following theorem.

\begin{theor} \label{thm: main-dual}
    Let $n$ be sufficiently large, and assume that $N$ satisfies
    \begin{equation} \label{eq: l_range_main dual}
        C_{\ref{thm: main}} n^2 \le N \le \exp\left(\frac{n}{C_{\ref{thm: main}}}\right).  
    \end{equation}
    Then there exists a convex body $K \subseteq \R^n$ with the following property: \\   
    for any vector $\xi' \in \R^n$ lying on the line segment connecting the barycenter of $K$ and its Santaló point, any polytope $P$ satisfying
    \begin{equation*}
        P \subseteq K -\xi' \subseteq \frac{n}{C_{\ref{thm: main}}\log N}P
    \end{equation*}
    must have at least $N$ facets.
\end{theor}
 As before, this bound matches the one obtained in \cite{Szarek2014}, albeit with a different center.

\section{Basic definitions and standard tools}
The volume of a convex body $L \subseteq \R^n$ will be denoted by $|L|$. 
The expectation of a random variable $Y$ is denoted by $\Exp Y$ and the probability of an event $V$ by $\Prob (V)$.
We denote the standard Euclidean norm by $\norm{\cdot}$ and its unit ball by $B_2^n$.

We collect several standard notions which will be  repeatedly used below; see, e.g.
\cite{Vershynin2018HDP,Ball1997,Grunbaum1960,Pisier1989} for background.
\begin{defi}[Support function]
    For a convex body $L \subseteq \R^n$, its support function is defined as
    \[
        h_L(y) := \sup_{x\in L} \langle x,y\rangle, \qquad y\in\R^n\,.
    \] 
\end{defi}
The support function of $L$ can be used to define its polar body:
    \[
        L^\circ
            = \{y\in\R^n:h_L(y)\le 1\}.
    \]

\begin{defi}[Gaussian mean width]\label{def:gaussian_mean_width}
 Define the Gaussian mean width of a compact set $L\subseteq\R^n$ as
\[
    w_G(L):=\Exp \sup_{x\in L}\langle x,G\rangle,
\]
where $G$ stands for the standard Gaussian vector in $\R^n$.
\end{defi}
The Gaussian mean width is linear with respect to the Minkowski addition. 


We will use a special system of quasi-orthogonal vectors whose existence is stated in the next lemma.
\begin{lemma}[Existence of a design]
    \label{lem: xi-configure}
    There exists an absolute constant $C_{\ref{lem: xi-configure}}>0$ such that for every sufficiently large $n$ and $2 \le m \le \exp\Big( \tfrac{n}{C_{\ref{lem: xi-configure}}}\Big)$, there exist vectors $y_1, \dots, y_m \in \R^n$ with the following properties:
    \begin{enumerate}
        \item  For all $i \in [m]$,
        \[
            \frac{1}{2}\sqrt{n} \le \norm{y_i} \le 2\sqrt{n}.
        \]
        \item For all distinct pairs $\{i,j\} \in \binom{[m]}{2}$,
        \[
            |\langle y_i, y_j \rangle| \le C_{\ref{lem: xi-configure}} \sqrt{n \log m}.
        \]
    \end{enumerate}
\end{lemma}

The proof shows that independent standard Gaussian vectors satisfy these properties with high probability. The argument relies on the standard Gaussian concentration inequality combined with a union bound. We include the details in the Appendix.


The following basic estimate of one-dimensional marginals will be frequently used  later.
\begin{lemma}[Maximal one-dimensional projections]
    \label{lem:max_projection}
    Let $v_1,\dots,v_m \in S^{n-1}$ and let $X$ be a random vector distributed according to any of the measures:
    \[
        N(0,I_n), \quad {\rm unif}(\sqrt{n}S^{n-1}), \quad \text{or} \quad {\rm unif}(\sqrt{n}B_2^n).
    \]
    Then
    \[
        \Prob\left( \max_{i\in[m]} |\langle v_i, X \rangle| \ge C_{\ref{lem:max_projection}} \sqrt{\log m}\right) \le m^{-10}.
    \]
\end{lemma}
    The proof uses the standard concentration argument and the union bound. Constant 10 is chosen only for convenience; a similar bound holds with any exponent on $m$ if we adjust $C_{\ref{lem:max_projection}}$. We defer the proof of Lemma \ref{lem:max_projection} to the Appendix.

We will also need the classical Urysohn's inequality; see, e.g., \cite{Pisier1989}.
\begin{lemma} \label{lem:urysohn}
For any $n$-dimensional convex body $L\subseteq \R^n$,
\[
    \frac{|L|}{|B_2^n|}
\le 
    \left(\frac{w_G(L)}{w_G(B_2^n)} \right)^n\,,
\]
and equality holds when $L$ is a Euclidean ball.
\end{lemma}

\section{Approximation by a polytope with few vertices}
We equip $\R^{n+1}$ with the standard orthonormal basis $\{e_0, \dots, e_n\}$, identifying $\R^n$ with the span of $\{e_1, \dots, e_n\}$. Since Theorem~\ref{thm: main} is asymptotic with universal constants, we may construct our example in dimension $n+1$ while using $n$ instead of $n+1$ in  estimates.

We introduce several parameters that will be used in the remainder of the paper.
\begin{defi}[Parameter Configuration]
    \label{def: parameter_conf}
    Let 
    \begin{equation} \label{def: C_config}
        C_{\ref{def: C_config}} = \max \big(C_{\ref{lem: xi-configure}}, 100 C_{\ref{lem:max_projection}} \big),
    \end{equation}
    where $C_{\ref{lem: xi-configure}}, C_{\ref{lem:max_projection}}$ are the constants appearing in Lemmas \ref{lem: xi-configure}, \ref{lem:max_projection}.
    Assume  that the integers $m \ge n \ge 1$ satisfy
    \begin{equation} \label{eq:m_range}
        C_{\ref{def: C_config}} \, n \le m \le \exp\left(\frac{n}{C_{\ref{def: C_config}}}\right)
    \end{equation}
    and define the parameter
    \begin{equation} \label{eq:Delta_def}
        \Delta := C_{\ref{def: C_config}} \sqrt{\log m}.
    \end{equation}
    Furthermore, take vectors $y_1, \ldots y_m$ that satisfy Lemma \ref{lem: xi-configure}  and set $x_j= \frac{1}{\Delta} y_j$ for $j \in [m]$.
    By Lemma \ref{lem: xi-configure}, these vectors satisfy the following conditions:
       \begin{enumerate}
        \item \emph{Norm bounds:} For all $i \in [m]$,
        \begin{equation} \label{eq: xi_norm}
            \frac{1}{2}\frac{\sqrt{n}}{\Delta} \le \norm{x_i} \le 2\frac{\sqrt{n}}{\Delta}.
        \end{equation}
        \item \emph{Quasi-orthogonality:} For all distinct pairs $\{i,j\} \in \binom{[m]}{2}$,
        \begin{equation} \label{eq: xi_innerProduct}
            |\langle x_i, x_j \rangle| \le \frac{\sqrt{n}}{\Delta}.
        \end{equation}
        We will call the system $x_1, \ldots, x_m$ quasi-orthogonal.
        This system will remain fixed for the rest of the paper.
    \end{enumerate}
\end{defi}
 A quasi-orthogonal system is the only random ingredient of our construction.
 Probabilistic arguments will appear only in the proofs.

\subsection{Construction of the body $K$}

The hard to-approximate body will be constructed as the convex hull of two convex symmetric bodies in $\R^n$ shifted along the $e_0$ axis to form facets of the body $K \subseteq \R^{n+1}$. 
This symmetry implies that the barycenter and the Santal\'{o} point of $K$ are located on the $e_0$ axis.
The first body $Q \subseteq \R^n$ forming the top facet will be the convex symmetric hull of $m=nN$ quasi-orthogonal vectors.
As mentioned in the Introduction, symmetric bodies can be approximated by polytopes with far fewer vertices than Theorem \ref{thm: main} suggests. Nevertheless, we will show that when the center of homothety lies close to the bottom facet of $K$, approximating $K$ may still require at least $N$ vertices. To place the barycenter and the Santal\'{o} point of $K$ near this facet, we choose it to be a convex body whose volume is much larger than that of $Q$.
On the other hand, the shape of the bottom facet has to be chosen to make sure that it does not interfere with the hardness of approximation estimate, see Lemma \ref{lem: hardness} for details.
The bottom facet is constructed to accommodate both of these requirements simultaneously.

We will now proceed with the detailed construction.


\begin{defi} \label{def:Q}    
    Let  $\{x_i\}_{i \in [m]}$ be a quasi-orthogonal system in $\R^n$. 
    Define the polytope
    \begin{equation} \label{eq:def_Q}
        Q := \operatorname{conv}\{\pm x_i : i \in [m]\} \subseteq \R^n.
    \end{equation}

    Additionally, for any parameter $t > 0$, define the auxiliary body
    \begin{equation} \label{eq:def_Qt}
        Q_t := \operatorname{conv}\left(\frac{1}{t}Q \cup B_2^n\right).
    \end{equation}
\end{defi}
In the case where $x_1, \ldots, x_m$ are scaled independent Gaussian vectors, $Q$ is often referred to as a Gluskin polytope, because a similar construction was used in a celebrated paper of Gluskin  \cite{Gluskin1988}.

\begin{lemma}[Mean width estimate]
    \label{lem:Q_meanWidth}
    For every parameter $t \ge 0.02$, the Gaussian mean width satisfies
    \[
        w_G(Q_t) \le (1 + m^{-3}) \, w_G(B_2^n).
    \]
    In particular, setting $t=0.02$, we have for all sufficiently large $n$:
    \[
        w_G(Q) \le t \, w_G(Q_t) \le 0.05 \, w_G(B_2^n).
    \]
\end{lemma}

\begin{proof}
    Let $G \sim N(0,I_n)$ be a standard Gaussian vector.
 Although the proof can be written in terms of $G$ alone, it is more convenient to replace it with its normalized copy, i.e., a random unit vector. 
 
    Using the polar decomposition $G = \|G\|_2 \Theta$, where $\Theta$ is uniformly distributed on $S^{n-1}$ and independent of $\|G\|_2$, we can write
    \begin{equation}  \label{eq: width Q_t}
        w_G(Q_t) =  \Exp \|G\|_2 \cdot \Exp h_{Q_t}(\Theta) = w_G(B_2^n) \, \Exp h_{Q_t}(\Theta).
    \end{equation}
    In view of \eqref{eq:def_Qt}, the support function of $Q_t$ has the form
    \[
        h_{Q_t}(\theta) = \max\left\{ \frac{1}{t}h_Q(\theta), \, 1 \right\} \quad \text{for all } \theta \in S^{n-1}.
    \]
    Let $X := \sqrt{n}\Theta \sim \text{unif}(\sqrt{n}S^{n-1})$. For each $i \in [m]$, define the unit vector $v_i := x_i/\|x_i\|$. 
    By \eqref{eq: xi_norm}, we have
    \[
        h_Q(\Theta) = \max_{i \in [m]} |\langle x_i, \Theta \rangle| 
        = \frac{1}{\sqrt{n}} \max_{i \in [m]} |\langle x_i, X \rangle|
        \le \frac{2}{\Delta} \max_{i \in [m]} |\langle v_i, X \rangle|.
    \]
    Consider the event $\mathcal{E} := \{ h_Q(\Theta) \ge t \}$. For any $t \ge 0.02$, Lemma~\ref{lem:max_projection} implies
    \[
        \Prob(\mathcal{E}) 
        \le \Prob\left( \frac{2}{\Delta} \max_{i \in [m]} |\langle v_i, X \rangle| \ge t \right)
        = \Prob\left( \max_{i \in [m]} |\langle v_i, X \rangle| \ge \frac{t\Delta}{2} \right)
        \le m^{-10}.
    \]

    We now estimate the expectation $\Exp h_{Q_t}(\Theta)$ by splitting it over $\mathcal{E}$ and $\mathcal{E}^c$:
    \begin{itemize}
        \item On $\mathcal{E}^c$, we have $\frac{1}{t}h_Q(\Theta) < 1$, so $h_{Q_t}(\Theta) = 1$.
        \item On $\mathcal{E}$, we use the crude bound $h_{Q_t}(\Theta) \le \frac{1}{t}h_Q(\Theta) \le \frac{1}{t} \max_i \|x_i\| \le \frac{2\sqrt{n}}{t\Delta}$.
    \end{itemize}
    Thus,
    \[
        \Exp h_{Q_t}(\Theta) 
        \le 1 \cdot \Prob(\mathcal{E}^c) + \frac{2\sqrt{n}}{t\Delta} \Prob(\mathcal{E})
        \le 1 + \frac{100\sqrt{n}}{\Delta} m^{-10}.
    \]
    Given the constraints on $m$ and $\Delta$, the term $\frac{100\sqrt{n}}{\Delta} m^{-10}$ is strictly less than $m^{-3}$ for sufficiently large $n$. Substituting this  into \eqref{eq: width Q_t} completes the proof.
\end{proof}

The polar body of $Q_t$ can be computed using standard duality rules:
\begin{equation} \label{eq:Qt_polar}
    Q_t^\circ = \left( \operatorname{conv}\left(\frac{1}{t}Q \cup B_2^n \right) \right)^\circ 
    = \left(\frac{1}{t}Q\right)^\circ \cap (B_2^n)^\circ 
    = tQ^\circ \cap B_2^n.
\end{equation}
Volumetrically, both $Q_t$ and its polar $Q_t^\circ$ are essentially indistinguishable from the Euclidean ball, as established by the following lemma.

\begin{lemma}[Volume estimates]
    \label{lem:Qsharp_vol}
    For any parameter $t \ge 0.02$ and sufficiently large $n$, we have
    \[
        (1-o_n(1))|B_2^n| \le |Q_t^\circ| \le |B_2^n| \le |Q_t| \le (1+o_n(1))|B_2^n|.
    \]
\end{lemma}

\begin{proof}
The central inequalities $|Q_t^\circ| \le |B_2^n| \le |Q_t|$ follow immediately from the inclusions $Q_t^\circ \subseteq B_2^n \subseteq Q_t$. We establish the outer bounds separately.

\step{Lower bound on $|Q_t^\circ|$} 
Since $Q_t^\circ = tQ^\circ \cap B_2^n$, it suffices to show that the volume of the removed portion is negligible:
    \[
        \frac{|B_2^n \setminus tQ^\circ|}{|B_2^n|} = \frac{| (\sqrt{n}B_2^n) \setminus (\sqrt{n}tQ^\circ) |}{|\sqrt{n}B_2^n|} = o_n(1).
    \]
    Let $Y \sim \text{unif}(\sqrt{n}B_2^n)$. Recalling that $Q^\circ = \{y : \max_i |\langle x_i, y \rangle| \le 1\}$, the complement condition $y \notin tQ^\circ$ corresponds to $\max_i |\langle x_i, y \rangle| > t$. Thus,
    \[
        \frac{| (\sqrt{n}B_2^n) \setminus (\sqrt{n}tQ^\circ) |}{|\sqrt{n}B_2^n|} 
        = \Prob\left( \max_{i\in[m]} |\langle Y, x_i \rangle| > t\sqrt{n} \right).
    \]
    Denote $v_i = x_i/\|x_i\|$. In view of \eqref{eq: xi_norm}, we have 
    \[
        |\langle Y, x_i \rangle| = \|x_i\| |\langle Y, v_i \rangle| \le \frac{2\sqrt{n}}{\Delta} |\langle Y, v_i \rangle|.
    \]
    Consequently, the inequality $|\langle Y, x_i \rangle| > t\sqrt{n}$ implies $|\langle Y, v_i \rangle| > \frac{t\Delta}{2}$.
    Since $t \ge 0.02$, the choice of $C_{\ref{def: C_config}}$ yields $\frac{t\Delta}{2} \ge C_{\ref{lem:max_projection}} \sqrt{\log m}$.
     Applying Lemma~\ref{lem:max_projection} (with the bound for the ball), we obtain
    \[
        \Prob\left( \max_{i\in[m]} |\langle Y, v_i \rangle| > \frac{t\Delta}{2} \right) \le m^{-10} = o_n(1),
    \]
    as $m \ge n$.

\step{Upper bound on $|Q_t|$}
    In view of Lemma \ref{lem:urysohn},
    \[
        \left(\frac{|Q_t|}{|B_2^n|}\right)^{1/n} \le \frac{w_G(Q_t)}{w_G(B_2^n)}.
    \]
    By Lemma~\ref{lem:Q_meanWidth}, the ratio on the right is bounded by $1+m^{-3}$. Therefore,
    \[
        |Q_t| \le (1+m^{-3})^n |B_2^n| \le \exp(n m^{-3}) |B_2^n|,
    \]
    where $n m^{-3} \to 0$ since $m \ge n$. 
    Thus, $|Q_t| \le (1+o_n(1))|B_2^n|$ as claimed.
\end{proof}

 Now, we construct the body which will be used to prove the hardness of the approximation.

\begin{defi} \label{def: k}
    We define the convex body $K \subseteq \R^{n+1}$ by
    \begin{equation} \label{eq: defK}
        K := \operatorname{conv}\big( e_0 + Q, \, \Qs \big),
    \end{equation}
    where we identify $Q$ and $\Qs$ with their embeddings in the hyperplane $\{x : \langle x, e_0 \rangle = 0\}$.
\end{defi}

    \label{rem:center_of_K}
    Since $Q$ and $\Qs$ are centrally symmetric in $\R^n$, the body $K$ is invariant under the reflection $(t,x) \mapsto (t,-x)$ on $\R \times \R^n$. Consequently, the barycenter of $K$, denoted $g(K)$, and its Santaló point, denoted $s(K)$, must lie on the $e_0$-axis.

\subsection{Proof of Theorem \ref{thm: main}}
We rely on three key lemmas regarding the geometry of $K$. 
It will be convenient to introduce the translated family
\[
    K(\eta) := K - \eta e_0 = \operatorname{conv}\big( (1-\eta)e_0 + Q, \; -\eta e_0 + \Qs \big).
\]
This shift will be used in the proof of Theorem \ref{thm: main}. To prove Theorem \ref{thm: main-dual}, we have to introduce a family of deformations of $K(\eta)$.
For $\kappa \ge 1$ set
\begin{equation}  \label{eq: def K(eta,kappa)}
    K(\eta, \kappa) := \operatorname{conv}\big( (1-\eta)e_0 + Q, \; -\kappa\eta e_0 + \kappa\Qs \big)
\end{equation}
for $\eta \in \mathbb{R}$ and $\kappa \ge 1$. This body is  a scaled version of $(K(\eta)^\circ+t e_0)^\circ$ for $\kappa=\kappa(t)$.

We state the first lemma in the form that will be used in both the proof of Theorem \ref{thm: main} and the proof of Theorem \ref{thm: main-dual}. 

\begin{lemma}[Hardness of Approximation]
    \label{lem: hardness}
   Let $\kappa > 0$ and 
     let $\eta$ be a number satisfying
    \[
        0 \le \eta \le \frac{1}{\sqrt{n}}.
    \]
    Any polytope $P$ such that
    \[
        {P \subseteq  K(\eta,\kappa):= {\rm conv}\big( (1-\eta)e_0 + Q, \; -\kappa\eta e_0 + \kappa\Qs \big)
        \subseteq \frac{n}{96\Delta^2 \max\{\kappa,1\}}P,}
    \]
    must have at least $ \frac{m}{n}$
    vertices.
\end{lemma}
We note that when $\kappa = 1$, the body in the middle is exactly $K - \eta e_0$.
\begin{lemma}[Location of the Barycenter]
    \label{lem: delta_centerOfMass}
    The barycenter of $K$ is located at $g(K) = \eta_g e_0$, where  $\eta_g$ satisfies
    \[
        \frac{1}{10(n+1)} \le \eta_g \le \frac{10}{n+1}.
    \]
\end{lemma}

\begin{lemma}[Location of the Santaló Point]
    \label{lem: delta_Santalopt}
    The Santaló point of $K$ is located at $s(K) = \eta_s e_0$, where $\eta_s$ satisfies
    \[
        0 \le \eta_s \le \frac{10}{n+1}.
    \]
\end{lemma}

We will derive Theorem \ref{thm: main} now and defer the  proofs of the key lemmas to Section \ref{sec: key lemmas}.

\begin{proof}[Proof of Theorem \ref{thm: main}]
    Given the integer $N$ satisfying \eqref{eq: l_range_main}, set $m$
    \[
        m = n N.
    \]
    We will choose $C_{\ref{thm: main}}$  sufficiently large relative  to ensure that $m$ satisfies the range requirements of Definition \ref{def: parameter_conf}. 

    \step{Step 1: Location of the center}
    Let $\xi$ be any point on the line segment connecting the barycenter $g(K)$ and the Santaló point $s(K)$. As we noticed before, both centers lie on the $e_0$-axis. Furthermore, Lemmas \ref{lem: delta_centerOfMass} and \ref{lem: delta_Santalopt} yield
     that $\xi = \eta e_0$ with
    \[
        0 \le \eta \le \max(\eta_g, \eta_s) \le  \frac{10}{n+1} \le \frac{1}{\sqrt{n}},
    \]
    where the last inequality holds for sufficiently large $n$. Thus, the condition on $\eta$ in Lemma \ref{lem: hardness} is satisfied.

    \step{Step 2: Hardness of approximation}
    Consider any polytope $P$ satisfying the approximation condition
    \[
        P \subseteq K - \xi \subseteq \frac{n}{C_{\ref{thm: main}}\log N} P.
    \]
    Recalling that $\Delta = C_{\ref{def: C_config}} \sqrt{\log m} \le 2 C_{\ref{def: C_config}} \sqrt{\log N}$, we can choose $C_{\ref{thm: main}}$ large enough such that
    \[
        \frac{n}{C_{\ref{thm: main}}\log N} \le \frac{n}{96\Delta^2}.
    \]

    Consequently, $P$ satisfies the hypothesis of Lemma \ref{lem: hardness} with $\kappa=1$.
    In view of this lemma, the number of vertices of $P$ is at least
    \[
        \frac{m}{n} = N,
    \]
    which completes the proof.
\end{proof}

\section{Proofs of the key lemmas}  \label{sec: key lemmas}
\subsubsection{Hardness of Approximation}

To analyze the complexity of polytopes that approximate $K(\eta)$, we introduce a set of test vectors designed to distinguish the vertices of the top face.
This top face is the same for all values of $\kappa$, so the definition of the test vectors will be independent of this parameter.
\begin{defi}[Test Vectors]
    \label{def: test_vectors}
    Fix a shift parameter $\eta \le \frac{1}{2}$. For each $i \in [m]$ and $\sigma \in \{\pm 1\}$, we define the \emph{primary vectors} $x_{i,\sigma}^+ \in K(\eta,\kappa)$ and the \emph{dual functional directions} $x_{i,\sigma}^-$ as follows:
    \begin{align*}
        x_{i,\sigma}^+ &:= (1-\eta) e_0 + \sigma x_i, \\
        x_{i,\sigma}^- &:= - \frac{2\sqrt{n}}{\Delta} e_0 + \sigma x_i.
    \end{align*}
\end{defi}

The following lemma establishes the separation properties of these vectors.

\begin{lemma}
    \label{lem: separation_prop}
    For sufficiently large $n$, the test vectors satisfy the following inner product estimates:
    \begin{enumerate}
        \item \emph{Diagonal dominance:} For every pair $(i, \sigma)$,
        \begin{equation} \label{eq: innerProductplusminus_match}
             \frac{n}{8\Delta^2} \le \langle x_{i,\sigma}^+, \, x_{i,\sigma}^- \rangle \le \frac{4n}{\Delta^2}.
        \end{equation}
        \item \emph{Off-diagonal orthogonality:} For distinct pairs $(i,\sigma_i) \neq (j,\sigma_j)$,
        \begin{equation} \label{eq: innerProductplusminus}
            \langle x_{i,\sigma_i}^+, \, x_{j,\sigma_j}^- \rangle \le 0.
        \end{equation}
    \end{enumerate}
\end{lemma}

\begin{proof}
    \step{1. Off-diagonal case}
    Consider different pairs $(i, \sigma_i) \neq (j, \sigma_j)$. 
    Suppose first that $i \neq j$.
    In view of \eqref{eq: xi_innerProduct},
    \[
        \langle x_{i,\sigma_i}^+, \, x_{j,\sigma_j}^- \rangle
        = -2(1-\eta)\frac{\sqrt{n}}{\Delta} + \sigma_i \sigma_j \langle x_i, x_j \rangle
        \le 0,
    \]
    where we used assumption $\eta \le 1/2$.
    In the case $i = j$ but $\sigma_i \neq \sigma_j$, this is straightforward. 

    \step{2. Diagonal case}
    For the matched pair, we have:
    \[
        \langle x_{i,\sigma}^+, \, x_{i,\sigma}^- \rangle = -2(1-\eta)\frac{\sqrt{n}}{\Delta} + \|x_i\|^2.
    \]
    Using \eqref{eq: xi_norm}, we derive the upper bound:
    \[
        \langle x_{i,\sigma}^+, x_{i,\sigma}^- \rangle \le \|x_i\|^2 \le \frac{4n}{\Delta^2}.
    \]
    For the lower bound, since $\Delta \ll \sqrt{n}$ for large $n$, we have:
    \[
        \langle x_{i,\sigma}^+, x_{i,\sigma}^- \rangle \ge -\frac{2\sqrt{n}}{\Delta} + \frac{n}{4\Delta^2} 
        = \frac{n}{\Delta^2} \left( \frac{1}{4} - \frac{2\Delta}{\sqrt{n}} \right) \ge \frac{n}{8\Delta^2}.
    \]
\end{proof}

\begin{proof}[Proof of Lemma \ref{lem: hardness}]
Let
    \[
        R := \frac{n}{96\Delta^2 \max\{\kappa,1\}},
    \]
    and let $P$ be a polytope with $N$ vertices $\{w_\alpha\}_{\alpha=1}^N$ satisfying $P \subseteq K(\eta) \subseteq R P$.

    \step{1. Convex decomposition of $P$'s vertices}
    Since $P \subseteq K(\eta,\kappa)$, every vertex $w_\alpha$ of $P$ lies in $K(\eta,\kappa)$. 
    Recalling that $K(\eta,\kappa) = \operatorname{conv}\big( (1-\eta)e_0 + Q, \; -\kappa\eta e_0 + \kappa\Qs \big)$, we can express each $w_\alpha$ as a convex combination of points from the top part $(1-\eta)e_0+Q$ and one point from the bottom part $-\kappa\eta e_0 + \kappa\Qs$.
    
    Specifically, for each $\alpha \in [N]$, there exist nonnegative coefficients $\lambda_{i,\sigma}^{(\alpha)}$ and $\lambda_{B}^{(\alpha)}$ with
    \[
        \sum_{i\in[m]}\sum_{\sigma\in\{\pm1\}}\lambda_{i,\sigma}^{(\alpha)}+\lambda_B^{(\alpha)}=1,
    \]
    and a vector $y^{(\alpha)} \in \Qs$ such that
    
    \begin{equation} \label{eq: w_alpha_decomp}
        w_\alpha = \sum_{i \in [m]} \sum_{\sigma \in \{\pm 1\}} \lambda_{i,\sigma}^{(\alpha)} x^+_{i,\sigma} + \kappa\lambda_{B}^{(\alpha)} \big(- \eta e_0 + y^{(\alpha)}\big).
    \end{equation}
    Here, the vectors $x^+_{i,\sigma}$ are the test vectors defined in Definition \ref{def: test_vectors}, which are vertices $(1-\eta)e_0 +Q$. 

    \step{2. Covering argument}
    We now exploit the inclusion $K(\eta,\kappa) \subseteq R P$.
    For each index $i \in [m]$, consider the test pair $(x_{i,+1}^+, x_{i,+1}^-)$. Since $x_{i,+1}^+ \in K(\eta,\kappa) \subseteq R P$, we get
    \[
        R \max_{\alpha \in [N]} \langle w_\alpha, x_{i,+1}^- \rangle
        = \max_{z \in R P} \langle z, x_{i,+1}^- \rangle
        \ge \langle x_{i,+1}^+, x_{i,+1}^- \rangle
        \ge \frac{n}{8\Delta^2}.
    \]
    Here, the equality uses $RP=\operatorname{conv}\{Rw_\alpha:\alpha\in[N]\}$ and the fact that a linear functional attains its maximum on a polytope at a vertex; the last inequality follows from Lemma~\ref{lem: separation_prop}.
    
    With substitution $R = \frac{n}{96\Delta^2 \max\{\kappa,1\}}$, this simplifies to:
    \[
        \max_{\alpha \in [N]} \langle w_\alpha, x_{i,+1}^- \rangle \ge \frac{1}{R} \frac{n}{8\Delta^2} = \frac{96\Delta^2 \max\{\kappa,1\}}{n} \cdot \frac{n}{8\Delta^2} = 12\max\{\kappa,1\}.
    \]
    
    Motivated by this, we say that a vertex $w_\alpha$ \emph{covers} an index $i \in [m]$ if
    \begin{equation} \label{eq: covering_def}
        \langle w_\alpha, x_{i,+1}^- \rangle \ge 12\max\{\kappa,1\}.
    \end{equation}
    The previous bound implies that every index $i \in [m]$ is covered by at least one vertex $w_\alpha$.

    \step{3. Upper bound on the inner product}
    Fix any pair $(\alpha, i)$. Substituting the convex decomposition \eqref{eq: w_alpha_decomp} into the inner product, we have:
    \begin{align*}
        \langle w_\alpha, x_{i,+1}^- \rangle
        &= \sum_{j \in [m]} \sum_{\sigma \in \{\pm 1\}} \lambda_{j,\sigma}^{(\alpha)} \langle x^+_{j,\sigma}, x_{i,+1}^- \rangle
        + \kappa \lambda_B^{(\alpha)} \langle -\eta e_0 + y^{(\alpha)}, \, x_{i,+1}^- \rangle.
    \end{align*}
    
    We bound the two parts separately:
    \begin{itemize}
        \item \textit{Sum over vertices:} By Lemma \ref{lem: separation_prop}, $\langle x^+_{j,\sigma}, x_{i,+1}^- \rangle \le 0$ for all $(j,\sigma) \neq (i,+1)$. Thus, we can drop all terms except the one matching $(i,+1)$:
        \[
            \sum_{j,\sigma} \lambda_{j,\sigma}^{(\alpha)} \langle x^+_{j,\sigma}, x_{i,+1}^- \rangle 
            \le \lambda_{i,+1}^{(\alpha)} \langle x^+_{i,+1}, x_{i,+1}^- \rangle 
            \stackrel{\eqref{eq: innerProductplusminus_match}}{\le} \lambda_{i,+1}^{(\alpha)} \frac{4n}{\Delta^2}.
        \]
        \item \textit{Bottom facet term:} Recalling that $x_{i,+1}^- = -\frac{2\sqrt{n}}{\Delta}e_0 + x_i$ and observing that $x_i, y^{(\alpha)} \perp e_0$, we compute
        \[
            \langle -\kappa \eta e_0 + \kappa y^{(\alpha)}, \, x_{i,+1}^- \rangle
            = \kappa \left(\eta \frac{2\sqrt{n}}{\Delta} + \langle y^{(\alpha)}, x_i \rangle \right).
        \]
        Because $\eta \le 1/\sqrt{n}$ and $\Delta \ge 1$, we have $\eta \frac{2\sqrt{n}}{\Delta} \le 2$.
        Since $y^{(\alpha)} \in \Qs 
        = Q^\circ \cap B_2^n 
        \subseteq Q^\circ$, we have $\langle y^{(\alpha)}, x_i \rangle \le 1$. Hence
        \[
            \kappa \lambda_B^{(\alpha)} \langle -\eta e_0 + y^{(\alpha)}, x_{i,+1}^- \rangle
            = \lambda_B^{(\alpha)}\langle -\kappa \eta e_0 + \kappa y^{(\alpha)}, x_{i,+1}^- \rangle
            \le 3\kappa \lambda_B^{(\alpha)}
            \le 3\kappa.
        \]
    \end{itemize}
    Consequently, we have a simplified bound:
    \begin{equation}\label{eq: w_alpha_upper_bound}
        \langle w_\alpha, x_{i,+1}^- \rangle \le \frac{4n}{\Delta^2} \lambda_{i,+1}^{(\alpha)} + 3\kappa.
    \end{equation}

    \step{4. Counting vertices}
    Let $S_\alpha \subseteq [m]$ denote the set of indices covered by the vertex $w_\alpha$, i.e.,
    \[
        S_\alpha := \{ i \in [m] : \langle w_\alpha, x_{i,+1}^- \rangle \ge 12\max\{\kappa,1\} \}.
    \]
    For any $i \in S_\alpha$, the bounds  \eqref{eq: covering_def} and \eqref{eq: w_alpha_upper_bound} imply
    \[
        12\max\{\kappa,1\} 
    \le 
        \frac{4n}{\Delta^2} \lambda_{i,+1}^{(\alpha)} + 3\kappa 
        \quad \implies \quad \lambda_{i,+1}^{(\alpha)} \ge \frac{9 \Delta^2}{4n}\,.
    \]
    Since $\sum_{i \in [m]} \lambda_{i,+1}^{(\alpha)} \le 1$, we can estimate the size of each covering set $S_\alpha$:
    \[
        1 \ge \sum_{i \in S_\alpha} \lambda_{i,+1}^{(\alpha)} \ge |S_\alpha| \frac{9 \Delta^2}{4n} \quad \implies \quad |S_\alpha| \le \frac{4n}{9  \Delta^2}.
    \]
    Recall from Step 2 that every index $i \in [m]$ is covered by at least one vertex $w_\alpha$. Therefore, the total number of vertices $N$ satisfies
    \[
        m \le \sum_{\alpha=1}^N |S_\alpha| \le N \frac{4n}{9\Delta^2}.
    \]
    Rearranging and using $\Delta \ge 1$, we conclude:
    \[
        N \ge \frac{9\Delta^2}{4} \frac{m}{n} \ge \frac{m}{n}
    \]
    which completes the proof.
\end{proof}
Although the last inequality in the proof appears wasteful, retaining the factor $\frac{9\Delta^2}{4}$ does not improve the bound of Theorem \ref{thm: main} since this bound depends on $\log N$.

\subsubsection{Location of the Centers}

{Also updated the proof to get both bounds on $\eta_g$.}
\begin{proof}[Proof of Lemma \ref{lem: delta_centerOfMass}]
    Let $g(K) = \eta_g e_0$ denote the barycenter of $K$. Since
    \[
        K=\operatorname{conv}(e_0+Q,\Qs)\subseteq\{x:0\le \langle x,e_0\rangle\le 1\},
    \]
    and $g(K)\in K$, we have $\eta_g\ge 0$.

    By Grünbaum's inequality, any half-space containing the barycenter must contain at least $1/e$ of the body's total volume. Consider the "upper" and "lower" half-space defined by the cut at $\eta_g$:
    \[
        K_{\ge \eta_g} := K \cap \{x : \langle x, e_0 \rangle \ge \eta_g\}
        \qquad 
        K_{\le \eta_g} := K \cap \{x : \langle x, e_0 \rangle \le \eta_g\}
    \]
    Then
    \begin{equation} \label{eq: grunbaum}
        \min\{ |K_{\ge \eta_g}|, |K_{\le \eta_g}| \} \ge \frac{1}{e} |K|.
    \end{equation}
    {We will show that if $\eta_g$ is outside the claimed range, then either $|K_{\ge \eta_g}|$ is too small or $|K_{\le \eta_g}|$ is too small, contradicting the above inequality. }

    \step{1. Lower bound on $|K|$}
    We approximate $K$ from below by a cone with base $\Qs$ (at height 0) and apex $e_0+x_1$ (a point in the top face). Since $\Qs$ and $e_0+Q$ are subsets of $K$, we have:
    \[
        |K| \ge \frac{1}{n+1} |\Qs| \cdot \text{height} = \frac{1}{n+1} |\Qs|.
    \]
    Using the volume bound $|\Qs| \ge (1-o_n(1))|B_2^n|$ from Lemma \ref{lem:Qsharp_vol}, we obtain for sufficiently large $n$:
    \begin{equation} \label{eq: K_vol_lower}
        |K| \ge \frac{1}{2(n+1)} |B_2^n|.
    \end{equation}

    \step{2. Upper bound on $|K_{\ge \eta_g}|$}
    The cross-section of $K$ at height $t \in [0,1]$ is given by the Minkowski sum:
    \[
        K_t := \{x \in \R^n : (t, x) \in K\} = (1-t)\Qs + tQ.
    \]
    The volume of the upper part is simply the integral of these cross-sections from $\eta_g$ to $1$:
    \[
        |K_{\ge \eta_g}| = \int_{\eta_g}^1 |K_t| \, dt.
    \]
    To bound $|K_t|$, we use Urysohn's inequality and the additivity of Gaussian mean width:
    \[
        w_G(K_t) = (1-t)w_G(\Qs) + t w_G(Q).
    \]
    Since $\Qs \subseteq B_2^n$, we have $w_G(\Qs) \le w_G(B_2^n)$. From Lemma \ref{lem:Q_meanWidth}, we have $w_G(Q) \le 0.05 w_G(B_2^n)$. Thus,
    \[
        w_G(K_t) \le  (1 - 0.95t) w_G(B_2^n).
    \]
    Applying Urysohn's inequality yields
    \[
        |K_t| \le \left(\frac{w_G(K_t)}{w_G(B_2^n)}\right)^n |B_2^n| \le (1 - 0.95t)^n |B_2^n|.
    \]
    Integrating this bound:
    \begin{align*}
        |K_{\ge \eta_g}| 
        &\le |B_2^n| \int_{\eta_g}^1 (1 - 0.95t)^n \, dt 
        = |B_2^n| \left[ \frac{-(1-0.95t)^{n+1}}{0.95(n+1)} \right]_{\eta_g}^1 
        \le \frac{1}{0.95(n+1)} (1 - 0.95\eta_g)^{n+1} |B_2^n|.
    \end{align*}
    where in the last step we dropped the nonnegative term
    $\frac{1}{0.95(n+1)}(1-0.95)^{n+1}|B_2^n|$.
    Similarly, for the lower part, we have
    \begin{align*}
        |K_{\le \eta_g}| 
        &\le |B_2^n| \int_0^{\eta_g} (1 - 0.95t)^n \, dt 
        = \frac{1}{0.95(n+1)} (1 - (1 - 0.95\eta_g)^{n+1}) |B_2^n|.
    \end{align*}

    \step{3. Contradiction}
    Combining the bounds from Step 1 and Step 2 with Grünbaum's inequality \eqref{eq: grunbaum}:
    \[
         \frac{|K_{\ge \eta_g}|}{|K|} 
        \le \frac{\frac{1}{0.95(n+1)} (1 - 0.95\eta_g)^{n+1} |B_2^n|}{\frac{1}{2(n+1)} |B_2^n|}
        = \frac{2}{0.95} (1 - 0.95\eta_g)^{n+1}.
    \]
    If we assume $\eta_g > \frac{10}{n+1}$, the right-hand side is bounded above by 
    \[
        \frac{2}{0.95} \left(1 - \frac{9.5}{n+1}\right)^{n+1} \approx \frac{2}{0.95} e^{-9.5} < \frac{1}{e}.
    \]
    This is a contradiction. Conversely, for the lower bound, if we assume $\eta_g = \frac{1}{10(n+1)}$, then our integral yields 
    $$
        \frac{|K_{\ge \eta_g}|}{|K|} 
        \le
        \frac{2}{0.95} (1-(1 - 0.95\eta_g)^{n+1}) \approx \frac{2}{0.95} (1 - e^{-0.095}) < \frac{1}{e}, 
    $$
    which is again a contradiction. Therefore, $\eta_g$ must lie in the claimed range. 
\end{proof}
Before proving Lemma \ref{lem: delta_Santalopt}, let us give a geometric interpretation of the polar body $(K-\eta e_0)^\circ$. First observe that
\begin{align*}
    (K-\eta e_0)^\circ
    = \big((1-\eta)e_0+Q\big)^\circ \cap \big(-\eta e_0+\Qs\big)^\circ\,.
\end{align*}
Write any $y\in\R^{n+1}$ as $y=y_0e_0+y_\perp$ with $y_\perp\perp e_0$. Then
     \[
       h_{(1-\eta)e_0+Q}(y)=(1-\eta)y_0+h_Q(y_\perp)
     \] 
     and so
    \[
       \big((1-\eta)e_0+Q\big)^\circ
        =\Big\{y:\ y_0\le \tfrac{1}{1-\eta}\ \text{ and }\ y_\perp\in\big(1-(1-\eta)y_0\big)Q^\circ\Big\}.
    \]
In other words, $\big((1-\eta)e_0+Q\big)^\circ$ is an unbounded cone with the apex at $\frac{1}{1-\eta}e_0$ and cross-section $Q^\circ$ at height $y_0=0$.
Similarly, we obtain
\[
    \big(-\eta e_0+\Qs\big)^\circ
    =\Big\{y:\ -\tfrac{1}{\eta} \le y_0\ \text{ and }\ y_\perp\in\big(1+\eta y_0\big)Q_1\Big\}.
\]
which is an unbounded cone with apex at $-\frac{1}{\eta}e_0$ and cross-section $Q_1$ at height $y_0=0$. The polar body $(K-\eta e_0)^\circ$ is the intersection of these two cones.
For simplicity, we denote the two cones as $\Cu$ and $\Cm$ respectively:
\begin{align}
    \label{eq: Cone_Cu_Cm}
    \Cu=\Cu(\eta):=&\Big\{y:\  y_0\le \tfrac{1}{1-\eta}\ \text{ and }\ y_\perp\in\big(1-(1-\eta)y_0\big)Q^\circ\Big\},\ \text{ and}\\
    \nonumber
    \Cm=\Cm(\eta):=&\Big\{y:\ -\tfrac{1}{\eta}\le y_0\ \text{ and }\ y_\perp\in\big(1+\eta y_0\big)Q_1\Big\}.
\end{align}

Looking first at the cross-section at height $y_0=0$, it is worth noting that, unlike $Q_1^\circ$, the body $Q^\circ$ is much larger than $Q_1$ in volume. Consequently, for a typical slice at height $y_0=h$ (as long as $h$ is not too close to the height $\frac{1}{1-\eta}$ of the apex of $\Cu$), one expects the constraint coming from $\Cm$ to be the dominant one. We now make this intuition precise.

\begin{lemma}\label{lem:Cmprime_dominance}
Let $\eta\in(0,1)$. Fix $t=0.02$ and define
\[
    \Cm' := \Big\{y:\ -\tfrac{1}{\eta}\le y_0\le 0.98\ \text{ and }\ y_\perp\in\big(1+\eta y_0\big)Q_t^\circ\Big\}.
\]
Then $\Cm'\subseteq\Cm$ and
\[
    |\Cm'|\ge (1-o_n(1))|\Cm \cap \{ y \,:\, y_0 \le 0.98 \}|.
\]
Moreover,
\[
    \Cm' \subseteq \Cu.
\]
\end{lemma}

\begin{proof}
The inclusion $\Cm'\subseteq\Cm$ follows from $Q_t^\circ\subseteq B_2^n\subseteq Q_1$.
The second claim follows from Lemma~\ref{lem:Qsharp_vol} to get $|Q_t^\circ|\ge(1-o_n(1))|Q_1|$ and integration of the cross-sections.
Finally, take $y=y_0e_0+y_\perp\in\Cm'$. Using $Q_t^\circ=tQ^\circ\cap B_2^n$, we have
\[
    y_\perp\in (1+\eta y_0)Q_t^\circ
    =(1+\eta y_0)\big(tQ^\circ\cap B_2^n\big)
    \subseteq t(1+\eta y_0)Q^\circ.
\]
Therefore, if $t(1+\eta y_0)\le 1-(1-\eta)y_0$, then $y \in \Cu$. Now, the inequality is equivalent to
\[    
    y_0\le \frac{1-t}{1-(1-t)\eta} = \frac{0.98}{1-0.98\eta}.
\]
Since $\frac{0.98}{1-0.98\eta}\ge 0.98$ for all $\eta\in(0,1)$, this holds for every $y\in\Cm'$. Therefore $\Cm'\subseteq\Cu$.
\end{proof}

\begin{proof}[Proof of Lemma \ref{lem: delta_Santalopt}]
Let $s(K)=\eta_s e_0$ and consider the shifted body $K(\eta_s)=K-\eta_s e_0$, whose Santal\'o point is at the origin. For simplicity, set $\eta:=\eta_s$.

It is standard (see, e.g., \cite{Pisier1989}) that $s(K(\eta))=0$ if and only if the barycenter of $K(\eta)^\circ$ is at the origin. Hence Grünbaum's inequality applied to $K(\eta)^\circ$ gives
\begin{equation} \label{eq:grunbaum_Kpolar}
    \frac{|K(\eta)^\circ \cap \{y : \langle y, e_0 \rangle \le 0\}|}{|K(\eta)^\circ|} \ge \frac{1}{e}.
\end{equation}

Recall that $K(\eta)^\circ=\Cu\cap\Cm$. Therefore
\[
    |K(\eta)^\circ\cap\{y_0\le 0\}|\le |\Cm\cap\{y_0\le 0\}|.
\]
Since for each $h\in[-\tfrac{1}{\eta},0]$ we have the cross-section
\[
    \Cm\cap\{y_0=h\}=(1+\eta h)Q_1,
\]
we obtain by integration
\begin{equation}\label{eq:Cm_negative_volume}
    |\Cm\cap\{y_0\le 0\}|=\int_{-1/\eta}^0 (1+\eta h)^n|Q_1|\,dh
    =\frac{|Q_1|}{\eta(n+1)}.
\end{equation}

On the other hand, by Lemma~\ref{lem:Cmprime_dominance},
\[
    |K(\eta)^\circ|\ge |\Cm'|.
\]
Using the explicit cross-sections
\[
    \Cm'\cap\{y_0=h\}=(1+\eta h)Q_t^\circ,\qquad h\in\Big[-\tfrac{1}{\eta},0.98\Big],
\]
we get
\begin{equation}\label{eq:Cmprime_volume}
    |\Cm'|=\int_{-1/\eta}^{0.98} (1+\eta h)^n|Q_t^\circ|\,dh
    =\frac{(1+0.98\eta)^{n+1}}{\eta(n+1)}\,|Q_t^\circ|.
\end{equation}

Combining \eqref{eq:grunbaum_Kpolar}, \eqref{eq:Cm_negative_volume}, and \eqref{eq:Cmprime_volume}, we arrive at
\[
    \frac{1}{e}
    \le \frac{|K(\eta)^\circ\cap\{y_0\le 0\}|}{|K(\eta)^\circ|}
    \le \frac{|Q_1|}{|Q_t^\circ|}\,(1+0.98\eta)^{-(n+1)}.
\]
By Lemma~\ref{lem:Qsharp_vol} (applied with $t=0.02$ and with $t=1$ for $Q_1$), we have $|Q_1|=(1+o_n(1))|B_2^n|$ and $|Q_t^\circ|=(1-o_n(1))|B_2^n|$, hence $|Q_1|/|Q_t^\circ|=1+o_n(1) \le 2$ for sufficiently large $n$. 
Thus,
\[
    \frac{1}{e}\le 2(1+0.98\eta)^{-(n+1)},
\]
Taking logarithms yields $(n+1)\log(1+0.98\eta)\le 1+o_n(1)$, which in particular implies $\eta\le 1$ for all sufficiently large $n$. Using $\log(1+u)\ge u/2$ for $u\in[0,1/2]$ (with $u=0.98\eta$), we conclude
\[
    \frac{n+1}{2}\,0.98\eta \le 1+o_n(1),
\]
and hence $\eta\le \frac{10}{n+1}$ for all sufficiently large $n$.
\end{proof}

\section{Approximation by a polytope with few facets}
The dual statement does not follow by simply polarizing Theorem~\ref{thm: main}: polarity does not commute with translation, so shifting the body before taking polars produces a genuinely different family of convex bodies. Nevertheless, the symmetry of our construction allows us to adapt the previous argument with only minor changes.

We will take the body in Theorem~\ref{thm: main-dual} to be
\[
    K_{\ref{thm: main-dual}} := K(\eta_g)^\circ.
\]
Since $g(K(\eta_g))=0$, the Santal\'o point of $K(\eta_g)^\circ$ is at the origin. 
The barycenter of this body is located on the $e_0$ axis because of the symmetry of the construction of $K$.
Denote
\[
    g\big(K(\eta_g)^\circ\big)=\gamma_g e_0.
\]
In the sequel we will need an explicit description of the polar of the shifted body $K(\eta)^\circ-h e_0$ (and later we will substitute $\eta=\eta_g$ and $h\in[0,\gamma_g]$).

\begin{lemma}\label{lem:polar_shift_identity}
Let $\eta\in[0,1)$ and let $h$ satisfy $- \frac{1}{\eta} \le h \le \frac{1}{1-\eta}$.
Set
\[
    \kappa:=\frac{1-(1-\eta)h}{1+\eta h}.
\]
Then
\[
    \big(K(\eta)^\circ-h e_0\big)^\circ
    =\frac{1}{1-(1-\eta)h}\,K(\eta,\kappa).
\]
\end{lemma}

\begin{proof}
Recall that $K(\eta)^\circ=\Cu(\eta)\cap\Cm(\eta)$. Since translation commutes with intersection and for convex sets containing the origin one has $(A\cap B)^\circ=\operatorname{conv}(A^\circ,B^\circ)$, we obtain
\[
    \big(K(\eta)^\circ-h e_0\big)^\circ
    =\operatorname{conv}\big((\Cu-h e_0)^\circ,\,(\Cm-h e_0)^\circ\big),
\]
where $\Cu=\big((1-\eta)e_0+Q\big)^\circ$ and $\Cm=\big(-\eta e_0+\Qs\big)^\circ$. Writing $y=y_0e_0+y_\perp$ and using the support-function representation, we can define $\Cu$ by the inequality
\[
    (1-\eta)y_0+h_Q(y_\perp)\le 1.
\]
Replacing $y$ by $y-h e_0$ yields
\[
    (1-\eta)y_0+h_Q(y_\perp)\le 1-(1-\eta)h.
\]
Equivalently, $\Cu-h e_0=(1-(1-\eta)h)\,\Cu$. Similarly, one checks that $\Cm-h e_0=(1+\eta h)\,\Cm$.

Taking polars and using scaling as well as the double polar identity $(A^\circ)^\circ=\operatorname{conv}(A\cup\{0\})$, we obtain
\begin{align*}
    (\Cu-h e_0)^\circ
    &=\frac{1}{1-(1-\eta)h}\,\operatorname{conv}\big((1-\eta)e_0+Q,\,0\big),
    \intertext{and}
    (\Cm-h e_0)^\circ
    &=\frac{1}{1+\eta h}\,\operatorname{conv}\big(-\eta e_0+\Qs,\,0\big).
\end{align*}
Since $Q$ and $Q_1^\circ$ contain the origin,  $1-\eta>0$,  $-\eta<0$, taking the convex hull of $(\Cu-h e_0)^\circ$ and $(\Cm-h e_0)^\circ$ amounts to taking the convex hull of the nonzero parts. Therefore,
\begin{align*}
    \big(K(\eta)^\circ-h e_0\big)^\circ
    &=\operatorname{conv}\Big(\frac{1}{1-(1-\eta)h}\big((1-\eta)e_0+Q\big),\,\frac{1}{1+\eta h}\big(-\eta e_0+\Qs\big)\Big)\\
    &=\frac{1}{1-(1-\eta)h}\,\operatorname{conv}\Big((1-\eta)e_0+Q,\,\kappa\big(-\eta e_0+\Qs\big)\Big)\\
    &=\frac{1}{1-(1-\eta)h}\,K(\eta,\kappa),
\end{align*}
which proves the claim.
\end{proof}

\begin{rem}
The proof of the next lemma follows the same scheme as the proof of Lemma~\ref{lem: delta_Santalopt}: one combines Grünbaum's inequality with the cone decomposition $K(\eta)^\circ=\Cu\cap\Cm$ and compares the relevant volumes via cross-section integration.
\end{rem}

The next lemma shows that the barycenter of the body $K(\eta_g)^\circ$ is not too far from its top vertex.
As before, pinpointing the location of the barycenter relies on Gr\"{u}nbaum's inequality.
\begin{lemma}\label{lem:gamma_g_bound}
    The barycenter of $K(\eta_g)^\circ$, $\gamma_g e_0$ satisfies
   $$\gamma_g \ge -10\log(2e)\,.$$
\end{lemma}

\begin{proof}
It is enough to prove the above inequality assuming that $\gamma_g <0$.
Set $\eta:=\eta_g$ and write $L:=K(\eta)^\circ$. By symmetry, $g(L)=\gamma_g e_0$ for some $\gamma_g\in\R$.
Consider the translated body $L-\gamma_g e_0$, whose barycenter is at the origin. Applying Grünbaum's inequality in direction $e_0$ yields
\begin{equation}\label{eq:grunbaum_gamma_g}
    \frac{|L\cap\{y_0\le \gamma_g\}|}{|L|}\ge \frac{1}{e}.
\end{equation}

We bound the numerator above using $L\subseteq \Cm(\eta)$:
\begin{align}\label{eq:gamma_g_numer}
    |L\cap\{y_0\le \gamma_g\}|
    &\le |\Cm\cap\{y_0\le \gamma_g\}|
      =\int_{-1/\eta}^{\gamma_g} (1+\eta h)^n|Q_1|\,dh
      =\frac{(1+\eta\gamma_g)^{n+1}}{\eta(n+1)}\,|Q_1|.
\end{align}

For the denominator, Lemma~\ref{lem:Cmprime_dominance} gives $\Cm'\subseteq \Cu\cap\Cm=L$, hence
\begin{align}\label{eq:gamma_g_denom}
    |L|\ge |\Cm'|
    =\int_{-1/\eta}^{0.98} (1+\eta h)^n|Q_t^\circ|\,dh
    =\frac{(1+0.98\eta)^{n+1}}{\eta(n+1)}\,|Q_t^\circ|.
\end{align}

Combining \eqref{eq:grunbaum_gamma_g}, \eqref{eq:gamma_g_numer}, and \eqref{eq:gamma_g_denom} yields
\[
    \frac{1}{e}
    \le \frac{|Q_1|}{|Q_t^\circ|}\left(\frac{1+\eta\gamma_g}{1+0.98\eta}\right)^{n+1}.
\]
By Lemma~\ref{lem:Qsharp_vol}, for all sufficiently large $n$ we have $\frac{|Q_1|}{|Q_t^\circ|}\le 2$. Therefore,
\[
    \left(\frac{1+\eta\gamma_g}{1+0.98\eta}\right)^{n+1}\ge \frac{1}{2e}
    \qquad\Longrightarrow\qquad
    1+\eta\gamma_g\ge (1+0.98\eta)\,(2e)^{-1/(n+1)}\ge (2e)^{-1/(n+1)}.
\]
Since $(2e)^{-1/(n+1)}=\exp\!\big(-\tfrac{\log(2e)}{n+1}\big)$ and $e^{-x}\ge 1-x$ for $x\ge 0$, we obtain
\[
    \eta\gamma_g\ge -\frac{\log(2e)}{n+1}.
\]
Using Lemma~\ref{lem: delta_centerOfMass}, we have $\eta\ge \frac{1}{10(n+1)}$ for all sufficiently large $n$, and hence
\[
    \gamma_g\ge -10\log(2e).
\]

\end{proof}

Lemmas \ref{lem:polar_shift_identity} and \ref{lem:gamma_g_bound} give us enough information to complete the proof of our second main result.
\begin{proof}[Proof of Theorem \ref{thm: main-dual}]
Fix an integer $N$ satisfying \eqref{eq: l_range_main dual} and set $m=nN$. As in the proof of Theorem~\ref{thm: main}, by choosing $C_{\ref{thm: main}}$ sufficiently large we may assume that $m$ satisfies the range requirements in Definition~\ref{def: parameter_conf}.
We take the body in Theorem~\ref{thm: main-dual} to be
\[
    \widetilde K:=K(\eta_g)^\circ.
\]
Since $g(K(\eta_g))=0$, we have $s(\widetilde K)=0$, and by with the notation introduced earlier, $g(\widetilde K)=\gamma_g e_0$.
Assume that $\xi$ lies on the line segment connecting these two points. Then $\xi=he_0$ for some
\[
    h\in[\min\{0,\gamma_g\},\,\max\{0,\gamma_g\}].
\]
Clearly, $\xi$ is an interior point of $\widetilde K$.
Let $P$ be a polytope satisfying
\begin{equation}\label{eq:dual_inclusion_assumption}
    P\subseteq \widetilde K-\xi\subseteq \lambda P,
    \qquad \text{where }\ \lambda:=\frac{n}{C_{\ref{thm: main}}\log N}.
\end{equation}
Since $0\in \operatorname{int}(\widetilde K-\xi)$ and $\widetilde K-\xi\subseteq \lambda P$, we get $0\in \operatorname{int}(P)$. Therefore $P^\circ$ is a bounded polytope, and the number of vertices of $P^\circ$ equals the number of facets of $P$.

Taking polars in \eqref{eq:dual_inclusion_assumption} yields
\begin{align}
    \label{eq:dual_polar_sandwich}
    \lambda^{-1}P^\circ\subseteq (\widetilde K-\xi)^\circ\subseteq P^\circ.
\end{align}
Now apply Lemma~\ref{lem:polar_shift_identity} with $\eta=\eta_g$ to obtain
\[
    (\widetilde K-\xi)^\circ
    =(K(\eta_g)^\circ-he_0)^\circ
    =\frac{1}{1-(1-\eta_g)h}\,K(\eta_g,\kappa),
\]
where
\[
    \kappa:=\frac{1-(1-\eta_g)h}{1+\eta_g h}.
\]
Scaling does not change the number of vertices, hence \eqref{eq:dual_polar_sandwich} implies that the polytope
\[
    P_1:= \big(1-(1-\eta_g)h\big)
\lambda^{-1} P^\circ    
\]
satisfies
\begin{equation}\label{eq:Keta_kappa_sandwich}
    P_1\subseteq K(\eta_g,\kappa)\subseteq \lambda\,P_1.
\end{equation}

We next bound $\max\{\kappa,1\}$ uniformly for $h$ on the segment between $0$ and $\gamma_g$.
If $h\ge 0$ then clearly $\kappa\le 1$.
If $h\le 0$, then $h\ge \gamma_g$ and Lemma~\ref{lem:gamma_g_bound} gives $\gamma_g\ge -10\log(2e)$. Using also Lemma~\ref{lem: delta_centerOfMass}, we have $\eta_g\le \frac{10}{n+1}$, so
for sufficiently large $n$,
\[
    1+\eta_g h  \ge \frac{1}{2}.
\]
for all sufficiently large $n$.
Therefore, for $h\le 0$,
\[
    \kappa=\frac{1-(1-\eta_g)h}{1+\eta_g h}
    \le 2(1+ 10\log(2e))\,.
\]

Finally, Lemma~\ref{lem: delta_centerOfMass} gives $0\le \eta_g\le \frac{10}{n+1}\le \frac{1}{\sqrt n}$ for large $n$,
so the hypotheses of Lemma~\ref{lem: hardness} apply to $K(\eta_g,\kappa)$.
Recalling that $\Delta=C_{\ref{def: C_config}}\sqrt{\log m}$ and $m=nN$, and using $N\ge C n^2$ from \eqref{eq: l_range_main},
we have $\log m=\log n+\log N\le 2\log N$ for large $n$, hence $\Delta^2\le C'\log N$ for a universal constant $C'$.
Thus, choosing $C_{\ref{thm: main}}$ sufficiently large (depending only on $C_\kappa$ and $C'$) ensures
\[
    \lambda=\frac{n}{C_{\ref{thm: main}}\log N}
    \le \frac{n}{96\Delta^2\max\{\kappa,1\}}.
\]
Applying Lemma~\ref{lem: hardness} to \eqref{eq:Keta_kappa_sandwich} yields that $P_1$ has at least $m/n=N$ vertices.
This concludes that $P$ has at least $N$ facets. 
\end{proof}

\appendix
\section{Standard concentration lemmas}
The proofs below rely on the standard Gaussian concentration inequality; see, e.g., \cite[Chapter~5]{Vershynin2018HDP}.
\begin{lemma}\label{lem:gaussian_concentration}
Let $G\sim N(0,I_n)$ and let $f:\R^n\to\R$ be $L$-Lipschitz with respect to $\|\cdot\|_2$.
Then for every $t\ge 0$,
\[
    \Prob\big(|f(G)-\Exp f(G)|\ge t\big)\le 2\exp\!\Big(-\frac{t^2}{2L^2}\Big).
\]
\end{lemma}

\begin{proof}[Proof of Lemma \ref{lem:max_projection}]
    We first remark that the statement of the lemma is trivial if $m=1$, as the probability bound $m^{-10}$ is $1$. Hence we may assume $m\ge 2$. In particular, $\log m >0$. Set $t := C_{\ref{lem:max_projection}} \sqrt{\log m}$. 
    
    \medskip
    \noindent \textit{Case 1: $X \sim N(0, I_n)$.} \\
    Fix $i \in [m]$. Since $v_i \in S^{n-1}$, the projection $\langle v_i, X \rangle$ follows a standard normal distribution $N(0,1)$. Applying the standard Gaussian tail bound yields
    \[
        \Prob\big(|\langle v_i, X \rangle| \ge t\big) \le 2\exp\left(-\frac{t^2}{2}\right).
    \]
    By the union bound over $i \in [m]$, we obtain
    \[
        \Prob\left(\max_{i\in[m]} |\langle v_i, X \rangle| \ge t\right)
        \le 2m \exp\left(-\frac{t^2}{2}\right)
        = 2 m^{1 - \frac{1}{2} C_{\ref{lem:max_projection}}^2}.
    \]
    Consequently, choosing $C_{\ref{lem:max_projection}}$ sufficiently large ensures the right-hand side is bounded by $m^{-10}$, as required.

\medskip
\noindent \textit{Case 2: $X \sim \text{unif}(\sqrt{n} S^{n-1})$.} \\
Let $G \sim N(0, I_n)$. We can represent $X$ as $X = \sqrt{n} G / \|G\|_2$. For any fixed $i \in [m]$, we then have
$$
    \langle X, v_i \rangle \sim \frac{\sqrt{n}g_1}{\sqrt{g_1^2 + \cdots + g_n^2}},  
$$
where $g_1, \dots, g_n$ are i.i.d. $N(0,1)$ random variables. Therefore,
\begin{align*}
    |\langle X, v_i \rangle| \ge t 
\qquad \Leftrightarrow \qquad
    \sqrt{n-t^2}|g_1| \ge t \sqrt{g_2^2 + \cdots + g_n^2} \,.
\end{align*}
If $t \ge \sqrt{n}$, this event is empty and its probability is zero. Hence we may assume $t < \sqrt{n}$. Conditioning on $g_2, \dots, g_n$ and applying the standard Gaussian tail bound gives
\begin{align*}
    \Prob\left(|g_1| \ge \frac{t}{\sqrt{n-t^2}} \sqrt{g_2^2 + \cdots + g_n^2}\,\bigg\vert\, g_2, \dots, g_n\right)
\le 
    2\exp\left(-\frac{t^2}{2(n-t^2)}(g_2^2 + \cdots + g_n^2)\right).
\end{align*}
Taking expectation and using $\mathbb{E}\exp(-\alpha g^2) = (1+2\alpha)^{-1/2}$ for $\alpha > 0$ and $g \sim N(0,1)$, we obtain
$$
\Prob\left(|g_1| \ge \frac{t}{\sqrt{n-t^2}} \sqrt{g_2^2 + \cdots + g_n^2}\right)
\le 2 \left(1+2\frac{t^2}{n-t^2}\right)^{-(n-1)/2} 
$$ 
For sufficiently large $C_{\ref{lem:max_projection}}$, we have $t = C_{\ref{lem:max_projection}}\sqrt{\log m} \ge 1$, and therefore
$$
    2 \left(1+2\frac{t^2}{n-t^2}\right)^{-(n-1)/2} 
\le 
    2\exp\left(-\frac{t^2}{2(n-t^2)}(n-1)\right)
\le 
    2\exp\left(-\frac{t^2}{4}\right)
= 
    2 m^{-\frac{1}{4} C_{\ref{lem:max_projection}}^2} = o(m^{-11}),
$$
provided $C_{\ref{lem:max_projection}}$ is sufficiently large. Finally, a union bound over $i \in [m]$ yields the desired estimate.

\medskip
\noindent \textit{Case 3: $X \sim \text{unif}(\sqrt{n} B_2^n)$.} \\
We can write $X = R U$, where $U$ is uniformly distributed on $S^{n-1}$ and $R = \|X\|_2 \in [0, \sqrt{n}]$.
Since $R \le \sqrt{n}$ almost surely, we have the pointwise inequality:
\[
    |\langle X, v_i \rangle| = R |\langle U, v_i \rangle| \le \sqrt{n} |\langle U, v_i \rangle|.
\]
Let $X' = \sqrt{n} U \sim \text{unif}(\sqrt{n} S^{n-1})$. The inequality above implies that the event $\{|\langle X, v_i \rangle| \ge t\}$ is contained in $\{|\langle X', v_i \rangle| \ge t\}$. Therefore, the probability bound established in Case 2 applies directly to Case 3.
\end{proof}

\begin{proof}[Proof of Lemma~\ref{lem: xi-configure}]
    Consider independent random vectors $X_1, \dots, X_m \sim N(0, I_n)$.
    
    First, by Lemma~\ref{lem:gaussian_concentration} (concentration of the norm) and the bound $m \le \exp(n/C_{\ref{lem: xi-configure}})$ from \eqref{eq:m_range}, a union bound implies that
    \[
        \|X_i\|_2 \in \left[\frac{1}{2}\sqrt{n}, 2\sqrt{n}\right] \quad \text{for all } i \in [m]
    \]
    with probability $1 - o_n(1)$, provided $C_{\ref{lem: xi-configure}}$ is sufficiently large.
    
    Similarly, for any fixed pair of indices $\{i,j\}$, we may condition on $X_i$ to see that $\langle X_i, X_j \rangle \sim \|X_i\|_2 Z$, where $Z \sim N(0,1)$ is independent of $X_i$. Standard Gaussian tail  bounds (or Lemma~\ref{lem:gaussian_concentration}) yield
    \[
        \Prob\left(|Z| \ge \frac{1}{2} C_{\ref{lem: xi-configure}}\sqrt{\log m}\right) \le 2\exp\left(-\frac{C_{\ref{lem: xi-configure}}^2 \log(m)}{8}\right)
        \le m^{-10},
    \]
    provided $C_{\ref{lem: xi-configure}}$ is sufficiently large. 
    Combining this with the event $\{\|X_i\|_2 \le 2\sqrt{n}\}$ (which holds with high probability as established above), we have
    \[
        |\langle X_i, X_j \rangle| = \|X_i\|_2 |Z| \le (2\sqrt{n}) \left(\frac{C_{\ref{lem: xi-configure}}\sqrt{\log m}}{2}\right) = \sqrt{n} C_{\ref{lem: xi-configure}} \sqrt{\log m}
    \]
    with probability $1- o(m^{-2})$. Taking a union bound over all $\binom{m}{2}$ pairs ensures this holds simultaneously for all distinct $\{i,j\}$ with probability $1 - o(1)$.
    
    Finally, set $x_i = \tfrac{X_i}{C_{\ref{lem: xi-configure}} \sqrt{\log m}}$ for each $i \in [m]$. The bounds established above imply that
    \[
        \|x_i\|_2 = \frac{\|X_i\|_2}{C_{\ref{lem: xi-configure}} \sqrt{\log m}} \in \left[\frac{\sqrt{n}}{2C_{\ref{lem: xi-configure}} \sqrt{\log m}}, \frac{2\sqrt{n}}{C_{\ref{lem: xi-configure}} \sqrt{\log m}}\right]
    \]
    and 
    \[
        |\langle x_i, x_j \rangle| = \frac{|\langle X_i, X_j \rangle|}{(C_{\ref{lem: xi-configure}} \sqrt{\log m})^2} \le \frac{\sqrt{n}}{C_{\ref{lem: xi-configure}} \sqrt{\log m}}
    \]
    simultaneously with high probability. Thus, such a configuration exists.
\end{proof}
\bibliographystyle{plain} 
\bibliography{ref}

\end{document}